\documentclass{amsart}

\usepackage{amssymb}
\usepackage{amsmath}
\usepackage{amsthm}

\newtheorem{theorem}{Theorem}

\newtheorem{definition}[theorem]{Definition}
\newtheorem{lemma}[theorem]{Lemma}

\usepackage[T1]{fontenc}

\usepackage{pgf,tikz}
\usetikzlibrary{arrows}
\usepackage{mathptmx}
\usepackage{helvet}
\usepackage{courier}
\usepackage{type1cm}
\usepackage{graphicx}
\usepackage{multicol}
\usepackage[bottom]{footmisc}
\usepackage{amsfonts}
\definecolor{uuuuuu}{rgb}{0.26666666666666666,0.26666666666666666,0.26666666666666666}
\definecolor{ffffff}{rgb}{1.0,1.0,1.0}
\definecolor{zzttqq}{rgb}{0.26666666666666666,0.26666666666666666,0.26666666666666666}
\definecolor{cqcqcq}{rgb}{0.7529411764705882,0.7529411764705882,0.7529411764705882}

\providecommand{\U}[1]{\protect\rule{.1in}{.1in}}

\begin{document}

\title[Low-discrepancy sequences]{Low-discrepancy sequences for piecewise smooth functions on 
the two-dimensional torus}

\author[L. Brandolini]{Luca Brandolini}

\address{Dipartimento di Ingegneria Gestionale, dell'Informazione e della Produzione,
                Universit\`{a} di Bergamo,
                Viale Marconi 5,
                 24044 Dalmine (BG), Italy.}
\email{luca.brandolini@unibg.it}  
             
            \author[L. Colzani]{Leonardo Colzani}
            \address{Dipartimento di Matematica e Applicazioni,
               Edificio U5,
               Universit\`{a} di Milano Bicocca,
               Via R.Cozzi 53, 20125 Milano, Italy.}
\email               {leonardo.colzani@unimib.it}
          
                \author[G. Gigante]{Giacomo Gigante} 
\address{                Dipartimento di Ingegneria Gestionale, dell'Informazione e della Produzione,
                Universit\`{a} di Bergamo,
                Viale Marconi 5,
                 24044 Dalmine (BG), Italy.}              
                 \email{giacomo.gigante@unibg.it}
                
                \author[G. Travaglini]{Giancarlo Travaglini}
\address {Dipartimento di Statistica e Metodi Quantitativi,
                 Edificio U7,
                 Universit\`{a} di Milano-Bicocca,
                 Via Bicocca degli Arcimboldi 8,
                 20126 Milano, Italy.}
\email      {giancarlo.travaglini@unimib.it}

\begin{abstract}
We produce explicit low-discrepancy infinite sequences which can be used
to approximate the integral of a smooth periodic function restricted to a convex
domain with positive curvature in $\mathbb{R}^{2}$. The proof depends on simultaneous diophantine
approximation and a general version of the Erd\H os-Tur\'an inequality.

\end{abstract}




\maketitle








\section{Introduction}

Let $f$ be a suitable function on the $d$-dimensional torus $\mathbb{T}%
^{d}=\mathbb{R}^{d}/\mathbb{Z}^{d}$, and let $\{t\left(  j\right)
\}_{j=1}^{N}$ be a distribution of points on $\mathbb{T}^{d}$. The quality of
the approximation of $\ \int_{\mathbb{T}^{d}}f(t)dt$ \ by the Riemann sum
$\ N^{-1}\sum_{j=1}^{N}f\left(  t\left(  j\right)  \right)  $ \ is a basic
problem with applications in 2D or 3D computer graphics, and also with
applications when $d$ is large (and the curse of dimensionality appears). See
e.g. \cite{CSA}. Any bound of the form%
\[
\left\vert \frac{1}{N}\sum_{j=1}^{N}f\left(  t\left(  j\right)  \right)
-\int_{\mathbb{T}^{d}}f(t)dt\right\vert \leq D\left(  \{t\left(  j\right)
\}_{j=1}^{N}\right)  V\left(  f\right)
\]
can be termed a {\it Koksma-Hlawka type
inequality}, provided the RHS is a {\it variation} $V(f)$ of the
function $f$ times a {\it discrepancy} $D\left(  \{t\left(  j\right)  \}_{j=1}%
^{N}\right)  $ of the finite set $\{t\left(  j\right)  \}_{j=1}^{N}$ with
respect to a reasonably simple family of subsets of $\mathbb{T}^{d}$.

The case $d=1$ is the amazingly simple Koksma inequality, where $\mathbb{T}$
is replaced by the unit interval, $V\left(  f\right)  $ is the usual total variation
and $D\left(  \{t\left(  j\right)  \}_{j=1}^{N}\right)  $ is the
*-discrepancy
\[
\sup_{0<\alpha\leq1}\left\vert \frac{1}{N}\sum_{j=1}^{N}\chi_{\left[
0,\alpha\right)  }\left(  t\left(  j\right)  \right)  -\alpha\right\vert \;,
\]
that is the discrepancy measured on the family of all intervals anchored at
the origin.

See \cite{BC}, \cite{CSA}, \cite{DP}, \cite{DT}, \cite{KN}, \cite{Matousek},
\cite{Travaglini} as general references.

The term Koksma-Hlawka inequality properly refers to E. Hlawka's
generalization of Koksma inequality to several variables, where $f$ is
required to have bounded variation in the sense of Hardy and Krause. In
one variable, many familiar bounded functions have bounded variation, but,
in several variables, the Hardy-Krause condition cannot be applied to most
functions with simple discontinuities. For example: the characteristic
function of a polyhedron has bounded Hardy-Krause variation if and only if the
polyhedron is a $d$-dimensional interval.

We recall some of the variants of the Koksma-Hlawka inequality which have
appeared in the literature so far. F. Hickernell \cite{Hickernell} has
proposed Koksma-Hlawka type inequalities for reproducing kernel Hilbert
spaces. J. Dick \cite{Dick} has used fractional calculus to prove a
Koksma-Hlawka type inequality for functions with relaxed smoothness
assumptions. G. Harman \cite{Harman} has considered a geometric approach and
measured the variation by counting the convex sets needed to describe
super-level sets of the function $f$. In \cite{BCGTsimplices} the authors of
the present paper proposed a Koksma-Hlawka type inequality especially
tailored for simplices, while in \cite{BCGTpiecewise} they have introduced a
Koksma-Hlawka type inequality for piecewise smooth functions.
Analogues of the above problem in more general settings can be found e.g. in
\cite{BCCGT} and \cite{BCCGST}.

K. Basu and A. Owen \cite{BO} have recently produced low-discrepancy sequences
for a triangle, where the discrepancy is the one considered in
\cite{BCGTsimplices}. In this paper we propose a sequence of points which
gives low discrepancy in the sense of the Koksma-Hlawka type inequality in
\cite{BCGTpiecewise}. We first recall a particular two-dimensional case of the statement therein.

\begin{theorem}
[\cite{BCGTpiecewise}] 
\label{T1}Let $h\left(  t\right)  =f\left(  t\right)  \chi
_{\Omega}\left(  t\right)  $, where $f$ is a smooth $\mathbb{Z}^{2}$-periodic
function on $\mathbb{R}^{2}$ and $\chi_{\Omega}$ is the characteristic
function of a bounded Borel set in $\mathbb{R}^{2}$. Let
\[
V(f):=4\left\Vert f\right\Vert _{L^{1}\left(  \mathbb{T}^{2}\right)
}+2\left\Vert \frac{\partial f}{\partial t_{1}}\right\Vert _{L^{1}\left(
\mathbb{T}^{2}\right)  }+2\left\Vert \frac{\partial f}{\partial t_{2}%
}\right\Vert _{L^{1}\left(  \mathbb{T}^{2}\right)  }+\left\Vert \frac
{\partial^{2}f}{\partial t_{1}\partial t_{2}}\right\Vert _{L^{1}\left(
\mathbb{T}^{2}\right)  }\;.
\]
Let $\left\{  t(j)\right\}  _{j=1}^{N}\subset\mathbb R ^{2}$, 
for any $s\in\left(  0,1\right)  ^{2}$ and for any $x\in\mathbb{R}^{2}$ let
\[
I\left(  s,x\right)  =\cup_{m\in\mathbb{Z}^{2}}\left(  \left[  0,s_{1}%
\right]  \times\left[  0,s_{2}\right]  +x+m\right),
\]
 and let%
\begin{equation}
D\left(  \left\{  t(j)\right\}  _{j=1}^{N}\right) 
 :=\sup_{s\in\left(  0,1\right)  ^{2},\;x\in\mathbb{R}^{2}}\left\vert
\frac{1}{N}\sum_{j=1}^{N}\sum_{m\in\mathbb{Z}^{2}}\chi_{I\left(  s,x\right)
\cap\Omega}\left(  t\left(  j\right)  +m\right)  -\left\vert   I\left(
s,x\right)    \cap\Omega\right\vert \right\vert \;. \label{dp kh}
\end{equation}
Then%
\[
\left\vert \frac{1}{N}\sum_{j=1}^{N}\sum_{m\in\mathbb{Z}^{2}}h\left(  t\left(
j\right)  +m\right)  -\int_{\mathbb{R}^{2}}h\left(  t\right)  ~dt\right\vert
\leq V(f)\ D\left(  \left\{  t(j)\right\}  _{j=1}^{N}\right)  \;.
\]

\end{theorem}

Observe that if a set $K\in\mathbb R^2$ does not intersect any of its integer translates, then it can be
thought of as a subset of $\mathbb T^2,$ and in that case the expression
\[
\frac1N\sum_{j=1}^{N}\sum_{m\in\mathbb{Z}^{2}}\chi_{K}
\left(  t\left(  j\right)  +m\right)-|K|
\]
 compares the measure of  $K$ with the share of 
points in $K$ of the collection obtained by projecting $\{t(j)\}_{j=1}^N$ onto $\mathbb T^2.$ 
It follows that the above theorem includes, but is slightly more general than the
analogous theorem where not only the function $f$ but also the set $\Omega$ and the point collection
$\{t(j)\}$ are in $\mathbb T^2$, and the quantity $D(\left\{t(j)\right\}_{j=1}^N)$ is just the discrepancy 
with respect to the intersection of $\Omega$ with all the rectangles in $\mathbb T^2$.

We are therefore interested in choices of the set $\left\{  t\left(  j\right)
\right\}  _{j=1}^{N}$ which give satisfactory upper bounds for the discrepancy
 (\ref{dp kh}).

An interesting result in this direction is due to J. Beck \cite{Beck}: for every positive integer $N$
there is a collection of $N$ points in the unit square with isotropic discrepancy (that is, the discrepancy with respect to all convex sets) bounded by
$cN^{-2/3}\log^4N.$ Since the discrepancy  (\ref{dp kh}) is smaller than the isotropic discrepancy, 
Beck's result gives a sequence that can be used in the Koksma-Hlawka type inequality in Theorem \ref{T1} when $\Omega$ is convex.
On the other hand, Beck's construction is somewhat intricate, and is obtained partly by random and partly by deterministic methods.

A more explicit extensible construction comes from a result of H. Niederreiter (see \cite{Nieder} or \cite[page 129 and page 132, Exercise 3.17]{KN}): if $1,\,\alpha,\,\beta$ are algebraic linearly independent on $\mathbb Q$, then the discrepancy of $\{(j\alpha,\,j\beta)\}_{j=1}^{N}$ with respect to all axis parallel rectangles 
contained in the unit square is bounded by $cN^{-1+\varepsilon}.$ This immediately implies that the isotropic
discrepancy of this sequence is bounded by
$cN^{-1/2+\varepsilon}$ (see \cite[Theorem 1.6, page 95]{KN}), an estimate that is far from Beck's result.


Our main result is the following.

\begin{theorem}
\label{HK} Assume that $\alpha,\beta$ are real algebraic numbers and that
$1,\alpha,\beta$ is a basis of a number field on $\mathbb Q$ of degree $3$. For all integers
$j\geq0$, let $\,\,t\left(  j\right)  =(j\alpha
,j\beta)$. Let $\Omega$ be a convex domain contained in
$\mathbb{R}^{2}$ with $\mathcal{C}%
^{2}$ boundary having strictly positive curvature. Then the
discrepancy defined in (\ref{dp kh}) satisfies
\begin{equation}
D\left(  \left\{  t(j)\right\}  _{j=1}^{N}\right)  \leq c\ N^{-2/3}\log N.
\label{dueterzi}%
\end{equation}
The above constant $c$ depends on the minimum and the maximum of the 
curvature of $\partial\Omega,$
on its length, and on the numbers $\alpha,\,\beta.$
\end{theorem}

For example, one can take $\alpha=\xi, \beta=\xi^2$, where $\xi$ is a real 
root of a third degree irreducible polynomial in $\mathbb Z.$ 

In other words, Theorem \ref{HK} says that a regularity assumption on the convex set $\Omega$
suffices for  the sequence in Niederreiter's result to improve Beck's estimate $N^{-2/3}\log^4 N$. This can be obtained 
by estimating directly the discrepancy $D(  \left\{  t(j)\right\}  _{j=1}^{N})$, and avoiding the isotropic discrepancy. The main tool that will allow us to do it is a version of the Erd\H{o}s-Tur\'{a}n inequality essentially contained in \cite{CGT}.

\section{Proofs and auxiliary results}

Let us begin by recalling the above mentioned general form of the Erd\H{o}s-Tur\'{a}n inequality.
\begin{theorem}\label{ET}
There exists a positive function $\psi\left(  u\right)  $ on $\left[
0,+\infty\right)  $ with rapid decay at infinity such that for every
collection of points $\left\{  t(j)\right\}  _{j=1}^{N}\subset\mathbb R  ^{d}$, for every bounded Borel set $D\subseteq\mathbb{R}^{d}%
$, and for every $R>0$,
\begin{align*}
&  \left\vert \frac{1}{N}\sum_{j=1}^{N}\sum_{m\in\mathbb{Z}^{2}}\chi_{D
}\left(  t\left(  j\right)  +m\right)  -\left\vert D\right\vert
\right\vert \\
&  \leq\left\vert \widehat{H}_{R}\left(  0\right)  \right\vert +\sum
_{n\in\mathbb{Z}^{2},\,0<|n|<R}\left(  \left\vert \widehat{\chi}_{D
}(n)\right\vert +\left\vert \widehat{H}_{R}(n)\right\vert \right)  \left\vert
\frac{1}{N}\sum_{j=1}^{N}e^{2\pi in\cdot t(j)}\right\vert .
\end{align*}
Here%
\[
H_{R}\left(  x\right)  =\psi(R\,\mathrm{dist}(x,\partial D)),
\]
where $\mathrm{dist}$ is the Euclidean distance in $\mathbb R^d.$
\end{theorem}

\begin{proof}
Take a smooth radial function $m\left(  \xi\right)$ supported in $\left\vert
\xi\right\vert <1/2$ and with $\int_{\mathbb{R}^{d}}m^{2}\left(
\xi\right)  d\xi=1,$ and define%
\begin{align*}
K\left(  x\right)   &  =\int_{\mathbb{R}^{d}}\left(  1+\left\vert
\xi\right\vert ^{2}\right)  ^{-\left(  d+1\right)  /2}\left( m\ast m\right)\left(
\xi\right)  e^{2\pi i\xi\cdot x}d\xi,\\
\psi\left(  u\right)   &  =e^{2\pi}\left(  \int_{\left\vert y\right\vert
\leq1}K\left(  y\right)  dy\right)^{-1}  \int_{\left\{  \left\vert y\right\vert
\geq u\right\}  }K\left(  y\right)  dy
\end{align*}
Since $\widehat K(\xi)=0$ if $|\xi|\ge1$, it follows from the Paley-Wiener theorem that $K(x)$ is an entire function of exponential
type smaller than $1$, positive with mean $1$, all its derivatives have rapid
decay at infinity, and $\left\vert \widehat{K}\left(  \xi\right)  \right\vert
\leq1$ for every $\xi\in\mathbb R^d.$ If we set $K_{R}\left(  x\right)  =R^{d}K\left(  Rx\right), $ 
then the functions
\begin{align*}
A\left(  x\right)   &  =\int_{\mathbb{R}^{d}}K_{R}\left(  y\right)
\left(  \chi_{D}\left(  x-y\right)  -H_{R}\left(  x-y\right)  \right)
dy\\
B\left(  x\right)   &  =\int_{\mathbb{R}^{d}}K_{R}\left(  y\right)
\left(  \chi_{D}\left(  x-y\right)  +H_{R}\left(  x-y\right)  \right)
dy,
\end{align*}
are entire of exponential type smaller than $R$ and%
\[
A\left(  x\right)  \leq\chi_{D}\left(  x\right)  \leq B\left(  x\right)
,\quad\left\vert B\left(  x\right)  -A\left(  x\right)  \right\vert \leq
4\psi\left(  R\mathrm{dist}(x,\partial D)/2\right)
\]
{(see \cite{CGT} for the details).}
Periodization gives%
\[
\sum_{m\in\mathbb{Z}^{d}}A\left(  x+m\right)  \leq\sum_{m\in\mathbb{Z}^{d}%
}\chi_{D}\left(  x+m\right)  \leq\sum_{m\in\mathbb{Z}^{d}}B\left(
x+m\right),
\]
{and, by the Poisson summation formula,}
\begin{align*}
\sum_{m\in\mathbb{Z}^{d}}A\left(  x+m\right)
 & =\sum_{n\in\mathbb{Z}^{d}}\widehat{K}\left(  R^{-1}n\right)
\left(  \widehat{\chi}_{D}\left(  n\right)  -\widehat{H}_{R}\left(
n\right)  \right)  e^{2\pi in\cdot x},\\
\sum_{m\in\mathbb{Z}^{d}%
}B\left(  x+m\right) & =\sum_{n\in\mathbb{Z}^{d}}\widehat{K}\left(  R^{-1}n\right)
\left(  \widehat{\chi}_{D}\left(  n\right)  +\widehat{H}_{R}\left(
n\right)  \right)  e^{2\pi in\cdot x}%
\end{align*}
are trigonometric polynomials of degree at most $R.$ It now follows that%
\begin{align*}
&  \frac{1}{N}\sum_{j=1}^{N}\sum_{m\in\mathbb{Z}^{d}}\chi_{D}\left(
t\left(  j\right)  +m\right)  -\left\vert D\right\vert \\
&  \leq\frac{1}{N}\sum_{j=1}^{N}\sum_{m\in\mathbb{Z}^{d}}B\left(  t\left(
j\right)  +m\right)  -\left\vert D\right\vert \\
&  =\frac{1}{N}\sum_{j=1}^{N}\sum_{n\in\mathbb{Z}^{d}}\widehat{K}\left(
R^{-1}n\right)  \left(  \widehat{\chi}_{D}\left(  n\right)  +\widehat{H}%
_{R}\left(  n\right)  \right)  e^{2\pi in\cdot t\left(  j\right)  }-\left\vert
D\right\vert \\
&  =\widehat{H}_{R}\left(  0\right)  +\sum_{n\in\mathbb{Z}^{d},0<\left\vert
n\right\vert <R}\widehat{K}\left(  R^{-1}n\right)  \left(  \widehat{\chi}_{D
}\left(  n\right)  +\widehat{H}_{R}\left(  n\right)  \right)  \frac{1}{N}%
\sum_{j=1}^{N}e^{2\pi in\cdot t\left(  j\right)  }\\
&  \leq\left\vert \widehat{H}_{R}\left(  0\right)  \right\vert +\sum
_{n\in\mathbb{Z}^{d},0<\left\vert n\right\vert <R}\left(  \left\vert
\widehat{\chi}_{D}\left(  n\right)  \right\vert +\left\vert
\widehat{H}_{R}\left(  n\right)  \right\vert \right)  \left\vert \frac{1}%
{N}\sum_{j=1}^{N}e^{2\pi in\cdot t\left(  j\right)  }\right\vert .
\end{align*}
Similar estimates from below can be proved, if one uses $A\left(  x\right)  $
instead of $B\left(  x\right)  .$
\end{proof}

A second tool in the proof is the estimate of the Fourier transform of
arcs of curves in  $\mathbb R^2$. The next two lemmas are well known (see e.g.
\cite[Chapter 8]{Stein}). We recall 
the proof of the first one 
both to help the unfamiliar reader, and to emphasize 
its two-dimensional nature.

{In what follows, for any arc $\gamma$ we will denote with
$\widehat\gamma(\xi)\,$ the Fourier transform of its arclength measure.} 

\begin{lemma}
\label{fourier_one}
Let $\Omega $ be a convex set in $\mathbb{R}^{2}$ with a $\mathcal{C}^{2}$ boundary
  with non-vanishing curvature. Let $\gamma $ be an
arc of $\partial \Omega $ and $\kappa_\mathrm {min} >0$ be the 
{minimum of the} curvature of $\gamma.$
Then for $\left\vert \xi \right\vert \geq 1,$ the Fourier transform is bounded by%
\[
\left\vert \widehat{\gamma }\left( \xi \right) \right\vert \leq \min \left( \ell,c\frac{ 1+\kappa_\mathrm {min}  ^{-1/2} }{%
|\xi| ^{1/2}}\right).
\]%
Here $\ell$ is the length of the arc and $c$ is a universal constant. 
\end{lemma}

\begin{proof}
Let $r\left( \tau\right) $ be the parametrization of $\gamma $ with respect to
arclength, so that 
\[
\widehat{\gamma }\left( \xi \right) =\int_{0}^{\ell}e^{-2\pi ir\left( \tau \right)
\cdot \xi }d\tau .
\]
For any $\xi $ we have the trivial estimate%
\[
\left\vert \int_{0}^{\ell}e^{-2\pi ir\left( \tau \right) \cdot \xi }d\tau \right\vert
\leq \ell.
\]%
Assume $\xi \neq 0$ and let 
\[
\xi =\rho \eta 
\]%
where $\left\vert \eta \right\vert =1$ and $\rho >0$. First consider the (at
most) 
{three} intervals $I_{1},\,I_2$ and $I_{3}$ where $\left\vert r^{\prime }\left(
\tau \right) \cdot \eta \right\vert >2^{-1/2}$. By Van der Corput's lemma, since 
$\left\vert r^{\prime }\left( \tau \right) \cdot \eta \right\vert >2^{-1/2}$ and the expression
$r^{\prime \prime }\left( \tau \right) \cdot \eta =-\kappa \left( \tau \right)
\nu\left( \tau \right) \cdot \eta $ changes sign at most once (here $\nu\left( \tau \right) $ 
and $\kappa(\tau )$
are respectively the
outer normal
and the curvature of $\gamma$ at a point  $r\left( \tau \right) $), then%
\[
\left\vert \int_{I_{i}}e^{-2\pi i\rho r\left( \tau \right) \cdot \eta
}d\tau \right\vert \leq \frac{c_{1}}{\rho }
\]%
($i=1,2,3$). The constant $c_{1}$ is universal. If $\left\vert r^{\prime }\left( \tau \right)
\cdot \eta \right\vert \leq 2^{-1/2}$ we have $\left\vert \nu\left( \tau \right)
\cdot \eta \right\vert \ge2^{-1/2}$ so that 
\[
\left\vert r^{\prime \prime
}\left( \tau \right) \cdot \eta \right\vert =\kappa \left( \tau \right) 
\left\vert \nu\left( \tau \right) \cdot \eta \right\vert >\kappa_{\mathrm{min}} 2^{-1/2}.
\]
Thus, by Van der Corput's lemma, for the at most 
{three} intervals $J_{1},\,J_2$ and $%
J_{3}$ where $\left\vert r^{\prime }\left( \tau \right) \cdot \eta \right\vert
\leq 2^{-1/2}$, we have%
\[
\left\vert \int_{J_{j}}e^{-2\pi i\rho r\left( \tau \right) \cdot \eta
}d\tau \right\vert \leq 
\frac{c_{2}}{\left( \kappa_\mathrm{min} \rho \right) ^{1/2}}
\]
($j=1,2,3$). Again, $c_{2}$ is a universal constant. Finally, 
\[
\left\vert \int_{0}^{\ell}e^{-2\pi i\rho r\left( \tau \right) \cdot \eta
}d\tau \right\vert \leq \min \left( \ell,\frac{3c_{1}}{\rho }+\frac{%
3c_{2}}{\left( \kappa_\mathrm {min}  \rho \right) ^{1/2}}\right) .
\]%
{When} $\rho \geq 1,$ this gives%
\[
\left\vert \int_{0}^{\ell}e^{-2\pi i\rho r\left( \tau \right) \cdot \eta
}d\tau \right\vert \leq \min \left( \ell,c\frac{1+\kappa_\mathrm {min}  ^{-1/2} }{%
\rho ^{1/2}}\right).
\]
\end{proof}

\begin{lemma}
\label{fourier_two}
The Fourier transform of the arclength measure on the  segment $\gamma $ joining two points $x$
and $y$ in $\mathbb{R}^{2}$ is%
\[
\widehat{\gamma }\left( \xi \right) =\left\vert x-y\right\vert \frac{\sin \left( \pi \left(
x-y\right) \cdot \xi \right) }{\pi \left( x-y\right) \cdot \xi }e^{-2\pi i%
\frac{\left( x+y\right) }{2}\cdot \xi }.
\]%
In particular, calling $\ell=\left\vert x-y\right\vert $ and $\theta =\frac{x-y}{%
\left\vert x-y\right\vert }$, we have%
\[
\left\vert \widehat{\gamma }\left( \xi \right) \right\vert \leq \min \left(
\ell,\frac{1}{\pi \left\vert \xi \cdot \theta \right\vert }\right) .
\]
\end{lemma}

\begin{proof}
This is just 
{an explicit} calculation. 
\end{proof}

Before we proceed with the proof of Theorem \ref{HK}, we need a few results on
convex sets in $\mathbb R^d$. Let us begin with some terminology.

\begin{definition}
\label{def}
Let $K$ be a non-empty compact convex subset (a ``convex body'') of $\mathbb R^d.$ The signed distance function
$\delta_K$ is defined by
\[
\delta_K(x)=\left\{
\begin{array}
[c]{ll}%
\mathrm{dist}(x,\,\partial K)  & \text{if }x\in K\\
-\mathrm{dist}(x,\,\partial K)  & \text{if }x\notin K.\\
\end{array}
\right.
\]
For any real number $u$, define%
\[
K^{u}=\left\{
x\in \mathbb R^d:\delta_K(x)\ge u
\right\}
\]
and%
\[
K_{u}=\left\{
x\in \mathbb R^d:\delta_K(x)= u
\right\}
\]
\end{definition}
The signed distance function is Lipschitz continuous with constant $1,$ and $|\nabla\delta_K|=1$ almost everywhere (see \cite[Section 14.6]{GT}).

\begin{definition}Let $B$ be the closed unit ball centered at the origin. If $K$ is a convex body in $\mathbb R^d,$ then the outer 
parallel body of $K$ at distance $r$ is defined as the Minkowski sum of  $K$ and $rB$,
\[
K+rB=\{x+y:x\in K,\,|y|\le r\},
\]
while the inner parallel body of $K$ at distance $r$ is defined as the Minkowski difference of $K$ and 
$rB,$
\[
K\div rB=\{x:x+rB\subset K\},
\]
\end{definition}

\begin{lemma}
\label{intersection}
Let $K$ be a convex body in $\mathbb{R}^{d}$.

\noindent (i) For any real number $u,$ the set
$K^u$ is the outer or the inner parallel body of $K$ at distance $|u|$, according to whether $u$ is negative 
or positive, that is
\begin{eqnarray*}
K^u&=&K+|u|B,\text{ if }u\le0,\\
K^u&=&K\div uB,\text{ if }u>0.
\end{eqnarray*}
(ii) For any real number $u,$ the set $K^u$ is convex (possibly empty).

\noindent (iii) If $M$ is a convex body too, then for every $u\ge0$,%
\[
\left(  M\cap K\right)  _{u}=\left(  M_{u}\cap K^{u}\right)  \cup\left(
M^{u}\cap K_{u}\right)  .
\]
\end{lemma}

\begin{proof}
Point (i) follows easily from the definitions, while the proof of (ii) can be found in \cite[Chapter 3]{S}.
As for point (iii), we sketch a proof, highlighting the main steps. First observe that $\partial (K^u)=K_u$ and that $(M\cap K)^u=M^u\cap K^u$ when $u\ge0.$ The thesis now follows after the observation that for any two compact sets $A$ and $B$ one has
\[
\partial (A\cap B)=(\partial A\cap B)\cup(A\cap\partial B).
\]
\end{proof}

\begin{lemma}\label{smooth}
Let $K$ be a convex body in $\mathbb R^d$ with $\mathcal C^2$ boundary and let 
$\kappa_\mathrm{max}$ be the maximum of all the principal curvatures of $\partial K.$  Finally, let 
\[
\Gamma=\Gamma(K,\kappa_\mathrm{max})=\{x:-(2\kappa_\mathrm{max})^{-1}<\delta_K(x)<(2\kappa_\mathrm{max})^{-1}\}.
\] 
Then $\delta_K\in\mathcal C^2(\Gamma).$ Furthermore, the level set $K_u$ is $\mathcal C^2$ whenever
$|u|<(2\kappa_\mathrm{max})^{-1}$ and its principal curvatures at a point $x$ are given by 
\[
\kappa_j(x)=\frac{\kappa_j(y)}{1-u\kappa_j(y)},\quad j=1,\ldots,d-1,
\]
where $y$ is the unique point of $\partial K$ such that $\mathrm{dist}(x,y)=|u|$ and 
$\kappa_j(y)$ are the principal curvatures of $\partial K$ at $y.$
\end{lemma}
\begin{proof}
This is essentially a reformulation of Lemmas 14.16 and 14.17 in \cite{GT} for the case of convex bodies.
\end{proof}

Let us now move back to the two-dimensional case. In the next two lemmas we estimate
the Fourier transforms of the functions $\chi_D$ and $H_R$ in Theorem \ref{ET}, for the specific 
type of sets $D$ that one needs in the proof of Theorem \ref{HK}.
\begin{lemma}
\label{acca}
Let $\Omega$ be a convex body in $\mathbb R^2$ with $\mathcal C^2$ boundary with non-vanishing curvature and let $\kappa_\mathrm{min}$ and $\kappa_\mathrm{max}$ be the minimum and the maximum of the curvature of $\partial \Omega$. Let $I$ be a rectangle
contained in a unit square with sides parallel to the axes, and call $K=\Omega\cap I$. Then there exists a constant $c$ depending 
only on $\kappa_\mathrm{min}$ such that for all $R\ge4\kappa_\mathrm{max}^2$ and for every  $n=(n_1,n_2)\in\mathbb Z^2$ with $0<|n|<R,$
\[
|\widehat H_R(n)|\le c\frac 1 {|n|^{3/2}}+c\frac 1{1+|n_1|}\frac 1{1+|n_2|}.
\]
Here, $H_R(x)$ is the function defined in Theorem \ref{ET} by
$H_R(x)=\psi(R|\delta_K(x)|).$ Finally, there is a universal constant $c>0$ such that 
for all $R\ge1,$ 
\[
|\widehat H_R(0)|\le \frac c R. 
\]
\end{lemma}
\begin{proof}
By the coarea formula (see \cite[Theorem 2, page 117]{EG}), since $|\nabla\delta_{K}(x)|=1$ almost everywhere,%
\begin{align*}
\widehat{H}_{R}(n)  &  =\int_{\mathbb{R}^{2}}\psi\left(  R|\delta_K(x)| \right)e^{-2\pi ix\cdot n}dx\\
&  =\int_{-\infty}^{\infty}\psi\left(  |Ru|\right)  \int_{K_{u}}e^{-2\pi
ix\cdot n}dxdu.
\end{align*}
where $K_{u}=\{x:\delta_K(x)=u\}$ as in the above Definition \ref{def}, and the integration on the level
set $K_{u}$ is with respect to the Hausdorff measure. Thus%
\begin{align*}
\left\vert \widehat{H}_{R}\left(  n\right)  \right\vert  &  \leq\int%
_{|u|<R^{-1/2}} \psi\left(  R|u|\right)   du\sup
_{\left\vert u\right\vert <R^{-1/2}}\left\vert \int_{K_{u}}e^{-2\pi ix\cdot
n}dx\right\vert \\
& \,\,\, +\int_{|u|\ge R^{-1/2}}\psi\left(  R|u|\right)
 \left\vert K_{u}\right\vert du\\
&  \leq\frac{c_{1}}{R}\sup_{\left\vert u\right\vert <R^{-1/2}}\left\vert
\int_{K_{u}}e^{-2\pi ix\cdot n}dx\right\vert +\frac{c_{2}}{R^{10}}.
\end{align*}
The constant $c_{1}$ is just the integral of $2\psi $ on
$\left[  0,+\infty\right),$ while $c_{2}$ depends on the rapid decay of
$\psi$ and the slow growth of $\left\vert K_{u}\right\vert $ (recall that
$K^{u}$ is convex and contained in a square of side $1+2|u|$, and therefore the Hausdorff measure of $K_u$ is smaller than
$4(1+2|u|)$). 
In particular, $c_1$ and $c_2$ are universal constants and we immediately have that for any  $R\ge 1$
\[
\left|\widehat H_R(0)\right|\le\frac c R,
\]
where  $c$ is a universal constant.

Now assume $n\neq 0,$ $R^{-1/2}\leq  1/(2\kappa_\mathrm{max})$ and $0\leq u\leq R^{-1/2}.$ Then, by the above Lemma
\ref{intersection} and Lemma \ref{smooth}, $K_{u}$ consists of at most
four smooth convex curves with curvature bounded below by $\kappa_\mathrm{min}$, and at
most four segments of length at most $1$ parallel to the axes. By Lemma \ref{fourier_one} and
Lemma \ref{fourier_two} this gives
\[
\sup_{0\leq u\leq R^{-1/2}}\left\vert \int_{K_{u}}e^{-2\pi ix\cdot
n}dx\right\vert \leq c\frac{1}{\left\vert n\right\vert
^{1/2}}+c\sum_{i=1}^{2}\frac{1}{1+\left\vert n_{i}\right\vert },
\]
where the constant $c$ depends only on the curvature $\kappa_\mathrm{min}.$ On the other
hand, if $R^{-1/2}\le 1/(2\kappa_\mathrm{max})  $ and if $-R^{-1/2}\le u<0$,
then $K_{u}$ is composed by at most four smooth convex curves with curvature
greater than or equal to $2\kappa_\mathrm{min}/3$, at most four segments parallel to the
axes and of length at most $1$, and at most eight arcs of circles of radius
$\left\vert u\right\vert $. In order to better understand this, observe (see Figure \ref{outer}) that
one can divide the complement of $K$ into at most sixteen regions by taking the two normals to 
$\partial K$ at each ``vertex'' of $K$ (there are at most eight ``vertices''). The part of
$K_u$ that intersects a region attached to a straight line is a parallel straight line of 
length at most $1$. The part of $K_u$ that intersects a region attached to a 
curve coming from $\partial \Omega$ is a part of $\Omega_u$. Finally, the part of  
$K_u$ that intersects a region attached to a vertex of $K$ is an arc of circle of radius $|u|$.

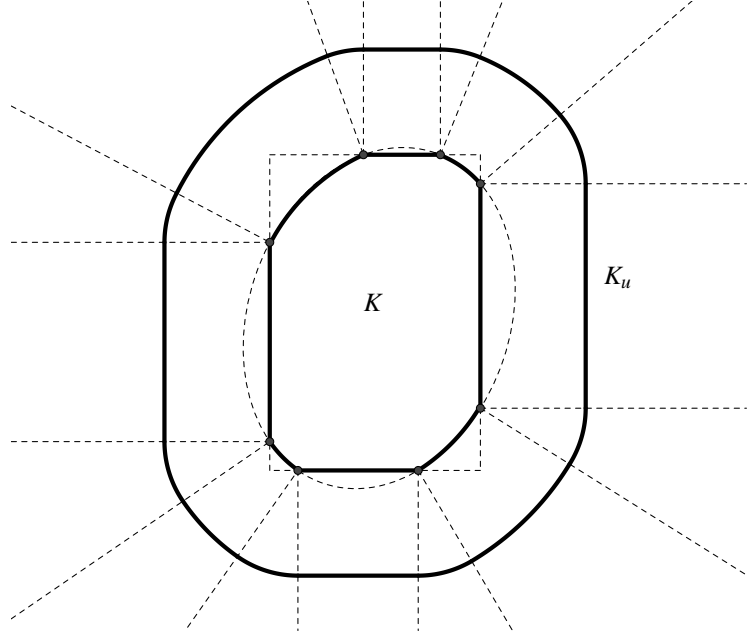
\begin{figure}
\centering
\begin{tikzpicture}[line cap=round,line join=round,>=triangle 45,x=7.0cm,y=7.0cm]
\clip(-0.2807182545241097,-0.09463180825196525) rectangle (1.1161071423012872,1.1006549587010719);
\fill[line width=0.0pt,dash pattern=on 3pt off 3pt,color=ffffff] (0.21043726626080655,0.20969973652946433) -- (0.21043726626080655,0.8096997365294644) -- (0.6104372662608073,0.8096997365294644) -- (0.6104372662608073,0.20969973652946433) -- cycle;
\draw [rotate around={71.36883352118411:(0.41832252974556966,0.49926781285205596)},line width=1.6pt,dash pattern=on 3pt off 3pt,color=ffffff,fill=ffffff,fill opacity=0.05] (0.41832252974556966,0.49926781285205596) ellipse (2.3211621356044287cm and 1.7400443885027044cm);
\draw [dash pattern=on 2pt off 2pt,domain=0.6104372662608073:1.1161071423012872] plot(\x,{(--0.1531003243861696-0.0*\x)/0.20289900496219448});
\draw [dash pattern=on 2pt off 2pt,domain=0.6104372662608073:1.1161071423012872] plot(\x,{(--0.07589559643367233-0.0*\x)/0.2314195233412687});
\draw [dash pattern=on 2pt off 2pt] (0.5344876935926546,0.8096997365294644) -- (0.5344876935926546,1.1006549587010719);
\draw [dash pattern=on 2pt off 2pt] (0.38834155253628333,0.8096997365294644) -- (0.38834155253628333,1.1006549587010719);
\draw [dash pattern=on 2pt off 2pt,domain=-0.2807182545241097:0.2104372662608065] plot(\x,{(-0.09681515941124277-0.0*\x)/-0.1505180222183522});
\draw [dash pattern=on 2pt off 2pt,domain=-0.2807182545241097:0.2104372662608065] plot(\x,{(-0.042945950165757536-0.0*\x)/-0.16240157154296653});
\draw [dash pattern=on 2pt off 2pt] (0.2638970169112267,0.20969973652946436) -- (0.2638970169112267,-0.09463180825196525);
\draw [dash pattern=on 2pt off 2pt] (0.4923562167699218,0.20969973652946436) -- (0.4923562167699218,-0.09463180825196525);
\draw [dash pattern=on 2pt off 2pt,domain=0.6104372662608073:1.1161071423012872] plot(\x,{(--0.035526988420498506--0.13544537159931247*\x)/0.15665716918351513});
\draw [dash pattern=on 2pt off 2pt,domain=0.5344876935926546:1.1161071423012872] plot(\x,{(-0.04059401368597726--0.19550970984756788*\x)/0.07892249104504068});
\draw [dash pattern=on 2pt off 2pt,domain=-0.2807182545241097:0.38834155253628333] plot(\x,{(-0.13414210497204693--0.1966149234640452*\x)/-0.07137011133276344});
\draw [dash pattern=on 2pt off 2pt,domain=-0.2807182545241097:0.2104372662608065] plot(\x,{(-0.1326687693042469--0.09283594785453886*\x)/-0.17588670257438746});
\draw [dash pattern=on 2pt off 2pt,domain=-0.2807182545241097:0.2104372662608065] plot(\x,{(-0.013300517014078993-0.07601553588083435*\x)/-0.11078767241749173});
\draw [dash pattern=on 2pt off 2pt,domain=-0.2807182545241097:0.2638970169112267] plot(\x,{(--0.014861321189561998-0.12761920351015116*\x)/-0.08973309279627881});
\draw [dash pattern=on 2pt off 2pt,domain=0.4923562167699218:1.1161071423012872] plot(\x,{(--0.07804585973471516-0.12675524813365224*\x)/0.07456912233983815});
\draw [dash pattern=on 2pt off 2pt,domain=0.6104372662608073:1.1161071423012872] plot(\x,{(--0.11865073819565669-0.10471108158755088*\x)/0.16688539143746883});
\draw [shift={(0.38810220184976074,0.38634545640533235)},line width=1.6pt]  plot[domain=3.742953962387106:4.0995643719694606,variable=\t]({1.0*0.21546473296783697*cos(\t r)+-0.0*0.21546473296783697*sin(\t r)},{0.0*0.21546473296783697*cos(\t r)+1.0*0.21546473296783697*sin(\t r)});
\draw [shift={(0.459578619043385,0.6241322102622126)},line width=1.6pt]  plot[domain=0.7129077649542921:1.1871253874080219,variable=\t]({1.0*0.19942626240059178*cos(\t r)+-0.0*0.19942626240059178*sin(\t r)},{0.0*0.19942626240059178*cos(\t r)+1.0*0.19942626240059178*sin(\t r)});
\draw [shift={(0.29258488213437456,0.5206986766476933)},line width=1.6pt]  plot[domain=5.283370185507661:5.7380818723313025,variable=\t]({1.0*0.3696335034824759*cos(\t r)+-0.0*0.3696335034824759*sin(\t r)},{0.0*0.3696335034824759*cos(\t r)+1.0*0.3696335034824759*sin(\t r)});
\draw [shift={(0.5468992102069045,0.46562282165221525)},line width=1.6pt]  plot[domain=2.00261196844457:2.655940078595787,variable=\t]({1.0*0.378852813315989*cos(\t r)+-0.0*0.378852813315989*sin(\t r)},{0.0*0.378852813315989*cos(\t r)+1.0*0.378852813315989*sin(\t r)});
\draw [line width=1.6pt] (0.5344876935926546,0.8096997365294644)-- (0.38834155253628333,0.8096997365294644);
\draw [line width=1.6pt] (0.2104372662608065,0.6432130716599225)-- (0.2104372662608065,0.2644429469353709);
\draw [line width=1.6pt] (0.2638970169112267,0.20969973652946436)-- (0.4923562167699218,0.20969973652946436);
\draw [line width=1.6pt] (0.6104372662608073,0.32795675722549544)-- (0.6104372662608073,0.7545641951999532);
\draw (0.37142056710623733,0.5627759135722052) node[anchor=north west] {$\huge K$};
\draw (0.8277025149136089,0.6145716734735035) node[anchor=north west] {$\huge K_u$};
\draw [rotate around={71.36883352118411:(0.41832252974556966,0.49926781285205596)},dash pattern=on 2pt off 2pt] (0.41832252974556966,0.49926781285205596) ellipse (2.3211621356044287cm and 1.7400443885027044cm);
\draw [shift={(0.5344876935926546,0.8096997365294644)},line width=1.6pt]  plot[domain=1.1871253874080208:1.5707963267948966,variable=\t]({1.0*0.20000000000000007*cos(\t r)+-0.0*0.20000000000000007*sin(\t r)},{0.0*0.20000000000000007*cos(\t r)+1.0*0.20000000000000007*sin(\t r)});
\draw [shift={(0.38834155253628333,0.8096997365294644)},line width=1.6pt]  plot[domain=1.5707963267948966:1.919000201977918,variable=\t]({1.0*0.20000000000000007*cos(\t r)+-0.0*0.20000000000000007*sin(\t r)},{0.0*0.20000000000000007*cos(\t r)+1.0*0.20000000000000007*sin(\t r)});
\draw [shift={(0.2104372662608065,0.6432130716599225)},line width=1.6pt]  plot[domain=2.655940078595787:3.141592653589793,variable=\t]({1.0*0.19999999999999993*cos(\t r)+-0.0*0.19999999999999993*sin(\t r)},{0.0*0.19999999999999993*cos(\t r)+1.0*0.19999999999999993*sin(\t r)});
\draw [shift={(0.2104372662608065,0.2644429469353709)},line width=1.6pt]  plot[domain=3.141592653589793:3.742953962387106,variable=\t]({1.0*0.2*cos(\t r)+-0.0*0.2*sin(\t r)},{0.0*0.2*cos(\t r)+1.0*0.2*sin(\t r)});
\draw [shift={(0.2638970169112267,0.20969973652946436)},line width=1.6pt]  plot[domain=4.0995643719694606:4.71238898038469,variable=\t]({1.0*0.20000000000000004*cos(\t r)+-0.0*0.20000000000000004*sin(\t r)},{0.0*0.20000000000000004*cos(\t r)+1.0*0.20000000000000004*sin(\t r)});
\draw [shift={(0.4923562167699218,0.20969973652946436)},line width=1.6pt]  plot[domain=4.71238898038469:5.244155315236654,variable=\t]({1.0*0.2*cos(\t r)+-0.0*0.2*sin(\t r)},{0.0*0.2*cos(\t r)+1.0*0.2*sin(\t r)});
\draw [shift={(0.6104372662608073,0.32795675722549544)},line width=1.6pt]  plot[domain=-0.5603541977995103:0.0,variable=\t]({1.0*0.19999999999999984*cos(\t r)+-0.0*0.19999999999999984*sin(\t r)},{0.0*0.19999999999999984*cos(\t r)+1.0*0.19999999999999984*sin(\t r)});
\draw [shift={(0.6104372662608073,0.7545641951999532)},line width=1.6pt]  plot[domain=0.0:0.7129077649542922,variable=\t]({1.0*0.19999999999999996*cos(\t r)+-0.0*0.19999999999999996*sin(\t r)},{0.0*0.19999999999999996*cos(\t r)+1.0*0.19999999999999996*sin(\t r)});
\draw [line width=1.6pt] (0.5344876935926546,1.0096997365294644)-- (0.38834155253628333,1.0096997365294644);
\draw [line width=1.6pt] (0.010437266260806488,0.6432130716599225)-- (0.010437266260806488,0.2644429469353709);
\draw [line width=1.6pt] (0.2638970169112267,0.009699736529464348)-- (0.4923562167699218,0.009699736529464348);
\draw [line width=1.6pt] (0.8104372662608073,0.32795675722549544)-- (0.8104372662608073,0.7545641951999532);
\draw [shift={(0.5217965660311035,0.48860101045602533)},line width=1.6pt]  plot[domain=1.9480108859673098:2.671652476569,variable=\t]({1.0*0.5475952611530206*cos(\t r)+-0.0*0.5475952611530206*sin(\t r)},{0.0*0.5475952611530206*cos(\t r)+1.0*0.5475952611530206*sin(\t r)});
\draw [shift={(0.45839717642482786,0.6250064016444956)},line width=1.6pt]  plot[domain=0.7093205227253653:1.18356620909317,variable=\t]({1.0*0.39975079966349847*cos(\t r)+-0.0*0.39975079966349847*sin(\t r)},{0.0*0.39975079966349847*cos(\t r)+1.0*0.39975079966349847*sin(\t r)});
\draw [shift={(0.31039574163824657,0.5094562637997624)},line width=1.6pt]  plot[domain=5.252946041399171:5.733229419476964,variable=\t]({1.0*0.5506495112107671*cos(\t r)+-0.0*0.5506495112107671*sin(\t r)},{0.0*0.5506495112107671*cos(\t r)+1.0*0.5506495112107671*sin(\t r)});
\draw [shift={(0.38772616271281013,0.384092862546494)},line width=1.6pt]  plot[domain=3.7389802182074434:4.0971771333595095,variable=\t]({1.0*0.41388349168498745*cos(\t r)+-0.0*0.41388349168498745*sin(\t r)},{0.0*0.41388349168498745*cos(\t r)+1.0*0.41388349168498745*sin(\t r)});
\draw [dash pattern=on 2pt off 2pt] (0.21043726626080655,0.20969973652946433)-- (0.21043726626080655,0.8096997365294644);
\draw [dash pattern=on 2pt off 2pt] (0.21043726626080655,0.8096997365294644)-- (0.6104372662608073,0.8096997365294644);
\draw [dash pattern=on 2pt off 2pt] (0.6104372662608073,0.8096997365294644)-- (0.6104372662608073,0.20969973652946433);
\draw [dash pattern=on 2pt off 2pt] (0.6104372662608073,0.20969973652946433)-- (0.21043726626080655,0.20969973652946433);
\begin{scriptsize}
\draw [fill=uuuuuu] (0.6104372662608073,0.7545641951999532) circle (1.5pt);
\draw [fill=uuuuuu] (0.6104372662608073,0.32795675722549544) circle (1.5pt);
\draw [fill=uuuuuu] (0.38834155253628333,0.8096997365294644) circle (1.5pt);
\draw [fill=uuuuuu] (0.5344876935926546,0.8096997365294644) circle (1.5pt);
\draw [fill=uuuuuu] (0.2104372662608065,0.2644429469353709) circle (1.5pt);
\draw [fill=uuuuuu] (0.2104372662608065,0.6432130716599225) circle (1.5pt);
\draw [fill=uuuuuu] (0.4923562167699218,0.20969973652946436) circle (1.5pt);
\draw [fill=uuuuuu] (0.2638970169112267,0.20969973652946436) circle (1.5pt);
\end{scriptsize}
\end{tikzpicture}
\caption{A convex body $K=\Omega\cap I$ and the relative set $K_u$, with $u<0.$ $\Omega$
has smooth boundary with non vanishing curvature and $I$ is a rectangle.}
\label{outer}
\end{figure}

It follows that%
\begin{align*}
\sup_{-R^{-1/2}\leq u<0}\left\vert \int_{K_{u}}e^{-2\pi ix\cdot n}%
dx\right\vert  &  \leq c\frac{1}{\left\vert n\right\vert
^{1/2}}+c\sum_{i=1}^{2}\frac{1}{1+\left\vert n_{i}\right\vert }+c\frac
{\left\vert u\right\vert^{1/2} }{ \left\vert
n\right\vert  ^{1/2}}\\
&  \leq c\frac{1}{\left\vert n\right\vert ^{1/2}}%
+c\sum_{i=1}^{2}\frac{1}{1+\left\vert n_{i}\right\vert },
\end{align*}
where the constant $c$ depends only on the minimal curvature $\kappa_\mathrm{min}.$
Therefore, when $0<\left\vert n\right\vert <R$ we have%
\begin{align*}
\left\vert \widehat{H}_{R}\left(  n\right)  \right\vert  &  \leq c\frac{1}%
{R}\frac{1}{\left\vert n\right\vert ^{1/2}}+c\sum_{i=1}%
^{2}\frac{1}{R}\frac{1}{1+\left\vert n_{i}\right\vert }+c\frac{1}{R^{10}}\\
&  \leq c\frac{1}{|n|^{3/2}}+c\frac{1}{1+|n_{1}|}\frac{1}{1+|n_{2}|}.
\end{align*}
\end{proof}

\begin{lemma}
\label{chi}
Let $\Omega$ be a convex body in $\mathbb R^2$ with $\mathcal C^2$ boundary with non-vanishing curvature and let $\kappa_\mathrm{min}$  be the minimum  of the curvature of $\partial \Omega$. Let $I$ be a rectangle
contained in a unit square with sides parallel to the axes, and call $K=\Omega\cap I$. Then there exists a constant $c$ depending 
only on $\kappa_\mathrm{min}$ such that for every  $n=(n_1,n_2)\in\mathbb Z^2\setminus\{(0,0)\}$
\[
|\widehat \chi_K(n)|\le c\frac 1 {|n|^{3/2}}+c\frac 1{1+|n_1|}\frac 1{1+|n_2|}.
\]
\end{lemma}
\begin{proof}
An application of the divergence theorem 
gives
\begin{eqnarray*}
\left\vert \widehat{\chi}_{K}(n)\right\vert 
&=& 
\left|\int_Ke^{-2\pi in\cdot x}dx\right|
= \left|\int_{\partial K}\frac{\nu(x)\cdot n}{2\pi i\left|n\right|^2}e^{-2\pi in\cdot x}\,dx   \right|\\
&=& \frac{1}{2\pi \left|n\right|}\left|\int_{\partial K}\nu(x)\cdot \frac{n}{\left|n\right|}e^{-2\pi in\cdot x}\,dx   \right|.
\end{eqnarray*}
Here $\nu(x)$ is the outer normal to $\partial K$ at the point $x$. 
This oscillatory integral can be estimated by means of standard techniques.
We include the details for the sake of completeness.
The boundary of $K=\Omega\cap I$ is composed by at most four smooth convex curves with
curvature bounded below by $\kappa_\mathrm{min}$, coming from $\partial\Omega$, and at most
four segments of length at most $1$ parallel to the axes, coming from
$\partial I$.
We therefore split the above integral into a sum of integrals over the 
components of $\partial K$ described above. 
When integrating over a segment, the quantity $\nu(x)\cdot n/|n|$ remains constant and
an immediate application of Lemma \ref{fourier_two} gives the estimate
\[
c\frac{1}{1+|n_{1}|}\frac{1}{1+|n_{2}|},
\]
with $c$ a universal constant. Let us now estimate the integral over an arc of $\partial \Omega,$ 
call it $\gamma.$ If $r(\tau )$ is a parametrization of $\gamma$ with respect to arclength, integration by parts gives
\begin{eqnarray*}
&&\left|\int_{\gamma}\nu(x)\cdot \frac{n}{\left|n\right|}e^{-2\pi in\cdot x}\,dx   \right|
=\left|\int_{0}^\ell \nu(r(\tau ))\cdot\frac n {|n|}e^{-2\pi i n\cdot r(\tau )}d\tau \right|\\
&=&\left| \nu(r(\ell))\cdot\frac n {|n|}\int_{0}^\ell e^{-2\pi i n\cdot r(u)}du  -
\int_0^\ell \frac d{d\tau}\left(\nu(r(\tau))\right)\cdot\frac n {|n|}\int_{0}^\tau e^{-2\pi i n\cdot r(u)}dud\tau \right|\\
&\leq&
\left| \int_{0}^\ell e^{-2\pi i n\cdot r(u)}du\right|  +
\int_0^\ell \kappa(\tau) d\tau\sup_{0\leq \tau\leq\ell}\left|\int_{0}^\tau e^{-2\pi i n\cdot r(u)}du \right|\\
&\leq&
\left| \int_{0}^\ell e^{-2\pi i n\cdot r(u)}du\right|  +
2\pi\sup_{0\leq \tau\leq\ell}\left|\int_{0}^\tau e^{-2\pi i n\cdot r(u)}du \right|
\leq \frac c{|n|^{1/2}}.
\end{eqnarray*}
Here $\kappa(\tau)$ is the curvature of $\gamma$ at the point $r(\tau)$ and 
$\int_0^\ell\kappa(\tau)d\tau$ is the total curvature of $\gamma.$ Since $\gamma$
is an arc of $\partial \Omega,$ the total curvature of $\gamma$ is smaller than 
the total curvature of $\partial \Omega,$ that is $2\pi.$
The last inequality is just an immediate application of Lemma \ref{fourier_one}, where
the constant $c$ above depends only on the minimal curvature $\kappa_\mathrm{min}$ of
$\partial\Omega.$ 

\end{proof}

We are now ready to proceed with the proof of the main result of the paper.

\begin{proof}[Proof of Theorem \ref{HK}]
Let $\kappa_\mathrm{min}$ and $\kappa_\mathrm{max}$ be the minimum and the maximum of the curvature of $\partial\Omega.$
If we call $m_{1},\ldots,m_{q}$ the lattice points for
which the sets
\[
\left(  \left[  0,s_{1}\right]  \times\left[  0,s_{2}\right]
+x+m_{i}\right)  \cap\Omega
\]
are nonempty, and let
\[
K_{i}=\left(  \left[  0,s_{1}\right]  \times\left[  0,s_{2}\right]
+x+m_{i}\right)  \cap\Omega,
\]
then of course%
\[
\cup_{m\in\mathbb{Z}^{2}}\left(  \left(  \left[  0,s_{1}\right]  \times\left[
0,s_{2}\right]  +x+m\right)  \cap\Omega\right)  =\cup_{i=1}^{q}K_{i}.%
\]
The number $q$ is bounded by the maximum number of unit squares with integer
vertices that intersect any given translate of $\Omega$ in $\mathbb{R}^{2}$. This
number is of course bounded by $\left(  \text{diam}\left(  \Omega\right)
+2\right)  ^{2}.$ We recall that we need a uniform estimate with respect to
$s$ and $x$. 

\noindent%

\begin{figure}[h]
\centering
\begin{tikzpicture}[line cap=round,line join=round,>=triangle 45,x=3.5cm,y=3.5cm]
\draw [color=cqcqcq,dash pattern=on 2pt off 2pt, xstep=3.5cm,ystep=3.5cm] (-0.3532345349542238,-0.29891451526399226) grid (2.468338239833443,2.2180566531545574);
\draw[->,color=black] (-0.3532345349542238,0.0) -- (2.468338239833443,0.0);
\foreach \x in {,1,2}
\draw[shift={(\x,0)},color=black] (0pt,2pt) -- (0pt,-2pt) node[below] {\footnotesize $\x$};
\draw[->,color=black] (0.0,-0.29891451526399226) -- (0.0,2.2180566531545574);
\foreach \y in {,1,2}
\draw[shift={(0,\y)},color=black] (2pt,0pt) -- (-2pt,0pt) node[left] {\footnotesize $\y$};
\draw[color=black] (0pt,-10pt) node[right] {\footnotesize $0$};
\clip(-0.3532345349542238,-0.29891451526399226) rectangle (2.468338239833443,2.2180566531545574);
\fill[color=zzttqq,fill=zzttqq,fill opacity=0.1] (0.4627777777777806,0.25666666666666693) -- (0.4627777777777806,0.7366666666666674) -- (0.7827777777777807,0.7366666666666674) -- (0.7827777777777807,0.25666666666666693) -- cycle;
\fill[color=zzttqq,fill=zzttqq,fill opacity=0.1] (1.4627777777777806,0.25666666666666693) -- (1.4627777777777806,0.7366666666666674) -- (1.7827777777777807,0.7366666666666674) -- (1.7827777777777807,0.25666666666666693) -- cycle;
\fill[color=zzttqq,fill=zzttqq,fill opacity=0.1] (0.4627777777777806,1.2566666666666668) -- (0.7827777777777807,1.2566666666666668) -- (0.7827777777777807,1.7366666666666672) -- (0.4627777777777806,1.7366666666666672) -- cycle;
\fill[color=zzttqq,fill=zzttqq,fill opacity=0.1] (1.7827777777777807,1.2566666666666668) -- (1.4627777777777806,1.2566666666666668) -- (1.4627777777777806,1.7366666666666672) -- (1.7827777777777807,1.7366666666666672) -- cycle;
\draw [color=zzttqq] (0.4627777777777806,0.25666666666666693)-- (0.4627777777777806,0.7366666666666674);
\draw [color=zzttqq] (0.4627777777777806,0.7366666666666674)-- (0.7827777777777807,0.7366666666666674);
\draw [color=zzttqq] (0.7827777777777807,0.7366666666666674)-- (0.7827777777777807,0.25666666666666693);
\draw [color=zzttqq] (0.7827777777777807,0.25666666666666693)-- (0.4627777777777806,0.25666666666666693);
\draw [color=zzttqq] (1.4627777777777806,0.25666666666666693)-- (1.4627777777777806,0.7366666666666674);
\draw [color=zzttqq] (1.4627777777777806,0.7366666666666674)-- (1.7827777777777807,0.7366666666666674);
\draw [color=zzttqq] (1.7827777777777807,0.7366666666666674)-- (1.7827777777777807,0.25666666666666693);
\draw [color=zzttqq] (1.7827777777777807,0.25666666666666693)-- (1.4627777777777806,0.25666666666666693);
\draw [color=zzttqq] (0.4627777777777806,1.2566666666666668)-- (0.7827777777777807,1.2566666666666668);
\draw [color=zzttqq] (0.7827777777777807,1.2566666666666668)-- (0.7827777777777807,1.7366666666666672);
\draw [color=zzttqq] (0.7827777777777807,1.7366666666666672)-- (0.4627777777777806,1.7366666666666672);
\draw [color=zzttqq] (0.4627777777777806,1.7366666666666672)-- (0.4627777777777806,1.2566666666666668);
\draw [color=zzttqq] (1.7827777777777807,1.2566666666666668)-- (1.4627777777777806,1.2566666666666668);
\draw [color=zzttqq] (1.4627777777777806,1.2566666666666668)-- (1.4627777777777806,1.7366666666666672);
\draw [color=zzttqq] (1.4627777777777806,1.7366666666666672)-- (1.7827777777777807,1.7366666666666672);
\draw [color=zzttqq] (1.7827777777777807,1.7366666666666672)-- (1.7827777777777807,1.2566666666666668);
\draw [rotate around={82.5441233851533:(1.198760039974897,1.0673394735378672)},fill=black,fill opacity=0.1] (1.198760039974897,1.0673394735378672) ellipse (3.0955553401224063cm and 2.5760334219450463cm);
\draw (1.175117384722429,1.9081111589543962) node[anchor=north west] {$\Omega$};
\draw (0.5498825084910711,0.7271119482951617) node[anchor=north west] {$K_1$};
\draw (1.4904067667536267,0.7271119482951617) node[anchor=north west] {$K_2$};
\draw (1.4850628789225895,1.5500706742748998) node[anchor=north west] {$K_3$};
\draw (0.5445386206600339,1.5393828986128253) node[anchor=north west] {$K_4$};
\end{tikzpicture}
\caption{The intersection of a convex set $\Omega$ with smooth boundary having non-vanishing curvature with the integer translates of a fixed rectangle.}
\label{two}
\end{figure}
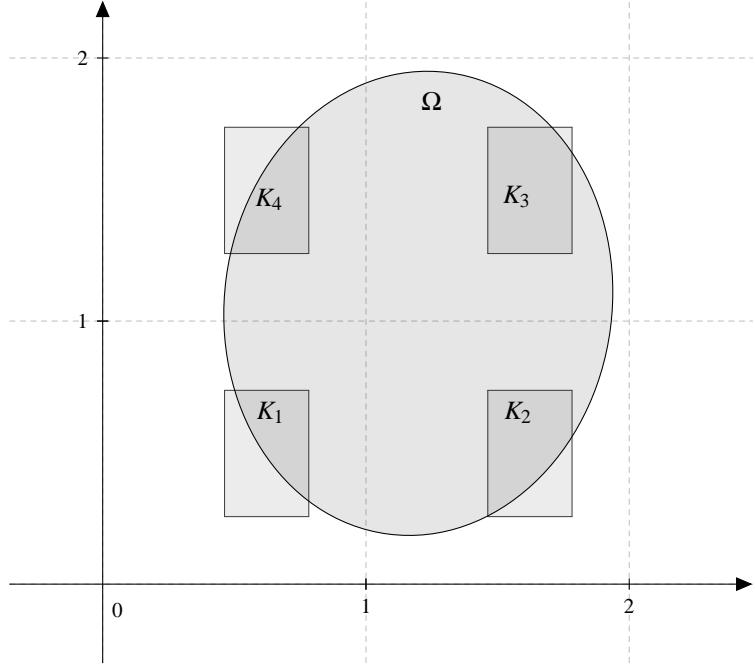

The sets $K_{i}$ are as in Figure \ref{two}; at most four sides are
parallel to the coordinate axes, while the curved parts come from
$\partial\Omega$. The discrepancy
\[
\left\vert \frac{1}{N}\sum_{j=1}^{N}\sum_{m\in\mathbb{Z}^{2}}\chi_{I\left(
s,x\right)  \cap\Omega}\left(  t\left(  j\right)  +m\right)  -\left\vert
  I\left(  s,x\right)   \cap\Omega\right\vert \right\vert
\]
is clearly bounded by the sum of the discrepancies of the sets $K_{i},$%
\[
\sum_{i=1}^{q}\left\vert \frac{1}{N}\sum_{j=1}^{N}\sum_{m\in\mathbb{Z}^{2}%
}\chi_{K_{i}}\left(  t\left(  j\right)  +m\right)  -\left\vert K_{i}%
\right\vert \right\vert ,
\]
and we shall therefore study the discrepancy of a single piece $K_{i}$. Let us call $K$ one such set.

By the general form of the Erd\H{o}s-Tur\'{a}n inequality in Theorem \ref{ET}, the discrepancy 
of a single piece $K$ is bounded by the
quantity
\begin{equation}
\left\vert \widehat{H}_{R}(0)\right\vert +\sum_{n\in\mathbb{Z}^{2}%
,\,0<|n|<R}\left(  \left\vert \widehat{\chi}_{K}(n)\right\vert +\left\vert
\widehat{H}_{R}(n)\right\vert \right)  \left\vert \frac{1}{N}\sum_{j=1}%
^{N}e^{2\pi in\cdot t(j)}\right\vert. \label{goal}%
\end{equation}
We recall that $R>0$ is a number that we can choose at our convenience, $H_{R}%
(x)=\psi(R|\delta_K(x)|)$ and $\psi(u)$ is a properly chosen function
on $[0,+\infty)$ with rapid decay at infinity.

The estimates of $\widehat \chi_K(n)$ and $\widehat H_R(n)$ are contained in the above Lemmas \ref{acca} and \ref{chi}, while
the estimate of the exponential sums follows a standard argument,%
\begin{align*}
&\left\vert \frac{1}{N}\sum_{j=1}^{N}e^{2\pi in\cdot\left(  j\alpha,j\beta\right)  }\right\vert =\left\vert
\frac{1}{N}\sum_{j=1}^{N}e^{2\pi ijn\cdot\left(  \alpha,\beta\right)
}\right\vert\\
 =&\left\vert \frac{1}{N}\frac{\sin\left(  \pi Nn\cdot\left(
\alpha,\beta\right)  \right)  }{\sin\left(  \pi n\cdot\left(  \alpha
,\beta\right)  \right)  }\right\vert \leq\frac{1}{N\left\Vert n\cdot\left(
\alpha,\beta\right)  \right\Vert },
\end{align*}
where $\left\Vert u\right\Vert $ is the distance 
{from $u$ to the closest integer}.

Overall, the goal estimate (\ref{goal}) becomes
\[
  \frac{1}{R}+\sum_{0<\left\vert n\right\vert <R}\left(  \frac{1}{|n|^{3/2}%
}+\frac{1}{1+|n_{1}|}\frac{1}{1+|n_{2}|}\right)  \frac
{1}{N\left\Vert n\cdot\left(  \alpha,\beta\right)  \right\Vert }.
\]
Observe now that
\begin{align*}
&\sum_{0<\left\vert n\right\vert <R}\frac{1}{ 1+\left\vert
n_{1}\right\vert  }\frac{1}{1+\left\vert n_{2}\right\vert
 }\frac{1}{\left\Vert n\cdot\left(  \alpha,\beta\right)  \right\Vert
}\\
 & \leq c\sum_{i=0}^{\log R}\sum_{j=0}^{\log R}\frac{1}{2^{i}}\frac{1}{2^{j}%
}\sum_{n_{1}=2^{i}}^{2^{i+1}-1}\sum_{n_{2}=2^{j}}^{2^{j+1}-1}\frac
{1}{\left\Vert n_{1}\alpha+n_{2}\beta\right\Vert }\\
& +c\sum_{i=0}^{\log R}%
\frac{1}{2^{i}}\sum_{n_{1}=2^{i}}^{2^{i+1}-1}\frac{1}{\left\Vert n_{1}%
\alpha\right\Vert }
 +c\sum_{j=0}^{\log R}\frac{1}{2^{j}}\sum_{n_{2}=2^{j}}^{2^{j+1}-1}\frac
{1}{\left\Vert n_{2}\beta\right\Vert }.
\end{align*}
Let us study the sum $\sum_{n_{1}=2^{i}%
}^{2^{i+1}-1}\sum_{n_{2}=2^{j}}^{2^{j+1}-1}\left\Vert n_{1}\alpha+n_{2}\beta\right\Vert ^{-1}$ first. By the
celebrated result of W. M. Schmidt \cite{SCH}, see also \cite[Theorem 7C]{SCH2}, since $1,\alpha,\beta$ are
linearly independent on $\mathbb{Q}$, for any $\varepsilon>0$ there is a
constant $\gamma>0$ such that for any $n\neq0$,
\begin{eqnarray}
\label{schmidt}
\left\Vert n_{1}\alpha+n_{2}\beta\right\Vert >\frac{\gamma}{\left(
1+\left\vert n_{1}\right\vert \right)  ^{1+\varepsilon}\left(  1+\left\vert
n_{2}\right\vert \right)  ^{1+\varepsilon}}.%
\end{eqnarray}
Then, following \cite{Davenport}, in any interval of the form
\[
\left[  \frac{\left(  k-1\right)  \gamma}{(1+2^{i+1})^{1+\varepsilon
}(1+2^{j+1})^{1+\varepsilon}},\frac{k\gamma}{(1+2^{i+1})^{1+\varepsilon
}(1+2^{j+1})^{1+\varepsilon}}\right)  ,
\]
where $k$ is a positive integer, there are at most two numbers of the form
$\left\Vert n_{1}\alpha+n_{2}\beta\right\Vert $, with $2^{i}\leq n_{1}%
<2^{i+1}$ and $2^{j}\leq n_{2}<2^{j+1}$. Indeed, assume by contradiction that
there are three such numbers. Then for two of them, say $\left\Vert
n_{1}\alpha+n_{2}\beta\right\Vert $ and $\left\Vert m_{1}\alpha+m_{2}%
\beta\right\Vert ,$ the fractional parts of $n_{1}\alpha+n_{2}\beta$ and
$m_{1}\alpha+m_{2}\beta$ belong either to $\left(  0,1/2\right]  $ or to
$\left(  1/2,1\right).$ Assume without loss of generality that they belong
to $\left(  0,1/2\right]  $. Then%
\begin{align*}
\frac{\gamma}{(1+2^{i+1})^{1+\varepsilon}(1+2^{j+1})^{1+\varepsilon}}  &
>\left\vert \left\Vert n_{1}\alpha+n_{2}\beta\right\Vert -\left\Vert
m_{1}\alpha+m_{2}\beta\right\Vert \right\vert \\
&  =\left\vert n_{1}\alpha+n_{2}\beta-p-\left(  m_{1}\alpha+m_{2}%
\beta-q\right)  \right\vert \\
&  \geq\left\Vert \left(  n_{1}-m_{1}\right)  \alpha+\left(  n_{2}%
-m_{2}\right)  \beta\right\Vert \\
& >\frac{\gamma}{(1+2^{i+1})^{1+\varepsilon
}(1+2^{j+1})^{1+\varepsilon}}.
\end{align*}
By the same type of argument, in the first interval $\left[
0,\frac{\gamma}{(1+2^{i+1})^{1+\varepsilon}(1+2^{j+1})^{1+\varepsilon}%
}\right),$ there are no points of the form $\left\Vert n_{1}\alpha
+n_{2}\beta\right\Vert $. It follows that%
\[
\sum_{n_{1}=2^{i}}^{2^{i+1}-1}\sum_{n_{2}=2^{j}}^{2^{j+1}-1}\frac
{1}{\left\Vert n_{1}\alpha+n_{2}\beta\right\Vert }\leq c\sum_{k=1}^{2^{i+j}%
}\frac{2^{\left(  i+j\right)  \left(  1+\varepsilon\right)  }}{k\gamma}\leq
c2^{\left(  i+j\right)  \left(  1+\varepsilon\right)  }\left(  i+j\right)  .
\]
Similarly,
\[
\sum_{n_{1}=2^{i}}^{2^{i+1}-1}\frac{1}{\left\Vert n_{1}\alpha\right\Vert }\leq
c2^{i\left(  1+\varepsilon\right)  }i,\qquad\sum_{n_{2}=2^{j}}^{2^{j+1}%
-1}\frac{1}{\left\Vert n_{2}\beta\right\Vert }\leq c2^{j\left(  1+\varepsilon
\right)  }j.
\]
Finally,%
\begin{align*}
&  \sum_{0<\left\vert n\right\vert <R}\frac{1}{\left(  1+\left\vert
n_{1}\right\vert \right)  }\frac{1}{\left(  1+\left\vert n_{2}\right\vert
\right)  }\frac{1}{\left\Vert n\cdot\left(  \alpha,\beta\right)  \right\Vert
}\\
&  \leq c\sum_{i=0}^{\log R}\sum_{j=0}^{\log R}\frac{1}{2^{i}}\frac{1}{2^{j}%
}2^{\left(  i+j\right)  \left(  1+\varepsilon\right)  }\left(  i+j\right)
+c\sum_{i=0}^{\log R}\frac{1}{2^{i}}2^{i\left(  1+\varepsilon\right)  }%
i+c\sum_{j=0}^{\log R}\frac{1}{2^{j}}2^{j\left(  1+\varepsilon\right)  }j\\
&  \leq c\sum_{i=0}^{\log R}2^{i\varepsilon}R^{\varepsilon}\log
R+cR^{\varepsilon}\log R\leq cR^{2\varepsilon}\log R.
\end{align*}
Finally, we use the hypothesis that  $1,\alpha,\beta$ are a basis of a number field 
in $\mathbb Q.$ By a simple argument in number field theory, there is a
constant $\eta$ such that for any $n\neq0$,
\[
\left\Vert n_{1}\alpha+n_{2}\beta\right\Vert >\frac{\eta}{\left(  \max\left(
\left\vert n_{1}\right\vert ,\left\vert n_{2}\right\vert \right)  \right)
^{2}}.
\]
See 
for example \cite[Theorem 6F]{SCH2}.
By a similar argument as before, this implies that
\[
\sum_{\max\left(  \left\vert n_{1}\right\vert ,\left\vert n_{2}\right\vert
\right)  =2^{i}}^{2^{i+1}-1}\frac{1}{\left\Vert n\cdot\left(  \alpha
,\beta\right)  \right\Vert }\leq c\sum_{k=1}^{2^{2i}}\frac{2^{2i}}{k}\leq
ci2^{2i}.
\]
Thus,
\begin{align*}
  \sum_{0<\left\vert n\right\vert <R}\frac{1}{\left\vert n\right\vert ^{3/2}%
}\frac{1}{\left\Vert n\cdot\left(  \alpha,\beta\right)  \right\Vert }
 & \leq c\sum_{i=0}^{\log R}\sum_{\max\left(  \left\vert n_{1}\right\vert
,\left\vert n_{2}\right\vert \right)  =2^{i}}^{2^{i+1}-1}\frac{1}{\left\vert
n\right\vert ^{3/2}}\frac{1}{\left\Vert n\cdot\left(  \alpha,\beta\right)
\right\Vert }\\
&  \leq c\sum_{i=0}^{\log R}\frac{1}{2^{3i/2}}i2^{2i}\leq cR^{1/2}\log R.
\end{align*}

Setting $R=N^{2/3}$ gives the desired estimate $N^{-2/3}\log N$, 
{as long as $N\ge 8\kappa_\mathrm{max}^3$}.
\end{proof}



\end{document}